\documentclass[10pt]{article}
\usepackage{indentfirst, latexsym, bm,amssymb,amsmath}
\usepackage{amsfonts,graphicx}
\usepackage{tikz}

\begin{document}

\newtheorem{theorem}{Theorem}[section]
\newtheorem{corollary}[theorem]{Corollary}
\newtheorem{definition}[theorem]{Definition}
\newtheorem{conjecture}[theorem]{Conjecture}
\newtheorem{question}[theorem]{Question}
\newtheorem{lemma}[theorem]{Lemma}
\newtheorem{proposition}[theorem]{Proposition}
\newtheorem{example}[theorem]{Example}
\newenvironment{proof}{\noindent {\bf
Proof.}}{\rule{3mm}{3mm}\par\medskip}
\newcommand{\remark}{\medskip\par\noindent {\bf Remark.~~}}
\newcommand{\pp}{{\it p.}}
\newcommand{\de}{\em}

\newcommand{\JEC}{{\it Europ. J. Combinatorics},  }
\newcommand{\JCTB}{{\it J. Combin. Theory Ser. B.}, }
\newcommand{\JCT}{{\it J. Combin. Theory}, }
\newcommand{\JGT}{{\it J. Graph Theory}, }
\newcommand{\ComHung}{{\it Combinatorica}, }
\newcommand{\DM}{{\it Discrete Math.}, }
\newcommand{\ARS}{{\it Ars Combin.}, }
\newcommand{\SIAMDM}{{\it SIAM J. Discrete Math.}, }
\newcommand{\SIAMADM}{{\it SIAM J. Algebraic Discrete Methods}, }
\newcommand{\SIAMC}{{\it SIAM J. Comput.}, }
\newcommand{\ConAMS}{{\it Contemp. Math. AMS}, }
\newcommand{\TransAMS}{{\it Trans. Amer. Math. Soc.}, }
\newcommand{\AnDM}{{\it Ann. Discrete Math.}, }
\newcommand{\NBS}{{\it J. Res. Nat. Bur. Standards} {\rm B}, }
\newcommand{\ConNum}{{\it Congr. Numer.}, }
\newcommand{\CJM}{{\it Canad. J. Math.}, }
\newcommand{\JLMS}{{\it J. London Math. Soc.}, }
\newcommand{\PLMS}{{\it Proc. London Math. Soc.}, }
\newcommand{\PAMS}{{\it Proc. Amer. Math. Soc.}, }
\newcommand{\JCMCC}{{\it J. Combin. Math. Combin. Comput.}, }
\newcommand{\GC}{{\it Graphs Combin.}, }
\title{On majorization of closed walks vector of trees with given degree sequences\thanks{
This work is supported by the National Natural Science Foundation of China (Nos. 11601337, 11601208 and 11531001 ), the Joint NSFC-ISF Research
Program (jointly funded by the National Natural Science Foundation of China and the Israel Science Foundation (No. 11561141001 )), Training program of
Lishui (No. 2014RC34) and the Foundation of Shanghai Normal University (No. SK201602) and the Foundation of Shanghai Municipal Education Commission .
\newline \indent $^{\dagger}$Correspondent author:
Ya-Lei Jin (Email: yaleijin@shnu.edu.cn)}}
\author{ Ya-Hong Chen $^{1}$, Daniel Gray$^2$,  Ya-Lei Jin$^{3\dagger}$, Xiao-Dong Zhang$^4$ 
\\
 {\small $^1$Department of Mathematics},
{\small Lishui University} \\
{\small  Lishui, Zhejiang 323000, PR China}\\
{\small $^2$Department of Mathematical Sciences, Georgia Southern University,}\\
{\small Statesboro, GA 30458, USA}\\
{\small $^3$Department of Mathematics, Shanghai Normal University,}\\
{\small 100 Guilin road, Shanghai 200234, PR China}\\
{\small $^4$School of Mathematical Sciences and MOE-LSC,}
{\small Shanghai Jiao Tong University} \\
{\small  800 Dongchuan road, Shanghai, 200240,  P.R. China}\\
}

\date{}
\maketitle
 \begin{abstract}
  Let $C_{v}(k;T)$ be the number of the closed walks of length $k$ starting at vertex $v$ in a tree $T$.  We prove that for a given  tree degree sequence $\pi$, then for any tree with degree sequence $\pi$, the sequence
   $C(k;T)\equiv(C_{v}(k;T), v\in V(T))$ is weakly majorized by the sequence $C(k;  T_{\pi}^*)\equiv C(k;  T_{\pi}^*, v\in V(T^*))$, where $T_{\pi}^*$ is the greedy tree corresponding to $\pi$. In addition, for two trees degree sequences $\pi,~\pi'$, if $\pi$ is majorized by $\pi'$, then $C(k;T_{\pi}^*)$ is weakly majorized by $C(k;T_{\pi'}^*)$.
   \end{abstract}

\section{Introduction}
Let $G=(V(G),E(G))$ be a simple graph of order $n$. A {\it walk} of $G$ is a sequence of vertices and edges, i.e., $w_1e_1w_2e_2\cdots e_{k-1}w_k$ such that $e_iw_iw_{i+1}\in E(G),~i=1,2,...,k-1$.  Moreover, if $w_1=w_k$, then this walk is called {\it closed walk} with length $k-1$. Further, denote by $C_{v}(k;G)$ be the number of the closed walks of length $k$ starting at vertex $v$ in  $G$ and the vector
   $C(k;G)\equiv(C_{v}(k;G), v\in V(G))$. Moreover, denote by $M_k(G)$ the number of the closed walks of length $k$  in $G$.  The number of closed walks of length $k$ in $G$ has been intensively studied. For example,  Dress etc. \cite{dress} studied when  $ M_{(k-1)}M_{k+1}(T)-M_{k}^2(T)$ is positive, zero, or negative.  Ta\"{u}big etc. \cite{taubig} investigate the growth of the number $M_k(G)$ and related inequalities.
   Further,  the number of closed walks  may be used to characterize  the complexity in the model of the symmetric Turing machine (see \cite{taubig}) and to study the Dense $r-$Subgraph Problem (see \cite{feige}). Since the dense $r-$subgraph maximization problem is of
computing the dense $r-$vertex subgraph of a given graph, it may be  an interesting problem to study the the number of closed walks of length $k$ with starting at vertices in any  vertex subset $U$ of $V(G)$ with $|U|=r\le n$. If $r=n,$  Csikvari \cite{Csikvari}  proved that the star has the maximum number of closed walks of length $k$ among all the trees on $n$ vertices, which confirm a conjecture of Nikiforov concerning the number of closed walks on trees. Further, Bollobas and Tyomkyn \cite{bob} proved that the $KC-$ transformation on tree increases the number of closed walks of length $k$.  In  addition,
  Andriantian and Wagner \cite{AnS} characterized the extremal trees with the maximum $M_k(T)$ among all trees with a given  tree degree sequence $\pi$.
  If $r<n$, there are no any results on the problem.

  On the other hand, the number of closed walks is direct relationship to the spectral radius of the adjacency matrix. Let $A(G)=(a_{ij})$ be the {\it adjacent matrix} of $G$, where $a_{ij}=1$ if $v_i$ is adjacent to $v_j$ and 0 otherwise,  then $A(G)$ has $n$ real eigenvalues $\lambda_1\ge \lambda_2\ge ...\ge \lambda_n$.  Since the trace of $A^k(G)$ is equal to the number of closed walks of length $k$ in $G$, it is easy to see that
\begin{equation}
M_k(G)=\sum_{i=1}^n\lambda^k,
\end{equation}
which is also called The {\it $k$-th spectral moment} of $G$.
Moreover, the  {\it Estrada index} \cite{PeG} of $G$, which is relative to $M_k(G)$ and proposed by Estrada, is defined to be
\begin{equation}
EE(G)=\sum_{i=1}^ne^{\lambda_i}.
\end{equation}
It is easy to see
\begin{eqnarray}
&&EE(G)=\sum_{i=1}^n\sum_{k=0}^{\infty}\frac{\lambda_i^k}{k!}=\sum_{k=0}^{\infty}\frac{M_k(G)}{k!}.
\end{eqnarray}
The Estrada index may have many applications  in the study of  molecular structures and  complex network, etc. For more about the Estrada index, the reader may refer to the excellent survey \cite{GuD}.
 A nonincreasing sequence of nonnegative integers
$\pi= (d_0, d_1,\cdots, d_{n-1})$ is called {\it graphic} if there exists a simple connected graph having $\pi$ as its vertex degree sequence.
  For a given tree degree  sequence $\pi=(d_0, d_1, \cdots, d_{n-1})$, let
$$\mathcal{T}_{\pi}=\{ T\ |\ T {\mbox{ is  any tree with}} \ \pi \ {\mbox{as  its degree sequence}}\}.$$
 There are several papers which investigated the graph parameters, such as
 Energy, Hosoya index and Merrifield-Simmons index in \cite{An}; the Estrada index in \cite{AnS}; the Wiener index in \cite{W1} and \cite{zhangxiang}; the largest spectral radius in \cite{BJ}; the Laplacian spectral radius in \cite{Zh}; the number of subtrees in \cite{ZhZ1, ZhZ2}, etc.

 In this paper, motivated by the  Dense $r-$Subgraph Problem and the research in the class $\mathcal{T}_{\pi}$, we  consider the following problem: determine
 $$\max_{T\in \mathcal{T}_{\pi}}\max_{U\subseteq V(T), |U|=r}\sum_{v\in U}C_k(v, T)$$
 for a given tree degree sequence $\pi$. The rest of this paper is arranged as follows. In Section 2, after introducing some notations, we present the main results of this paper. In the sections $3$ and $4$,  the proofs of Theorems \ref{Maintheorem1} and \ref{Maintheorem2} are given respectively.

\section{Preliminary and main results}
In order to present our main results, we first introduce  some notations. Let $G=(V(G), E(G))$ be a simple graph with a root set $V_0= \{v_{01}, ..., v_{0r}\}\subseteq V(G)$.  The {\it height} $h(v)$ of a vertex $v$ in $G$ is defined by
$$h(v)=dist(v, V_0)=\min_{w\in V_0}\{dist(v, w)\},$$
where $dist(v, w)$ is the distance between vertices $v$ and $w$ in $V(G)$. Moreover, we say that $v$ is at the $h(v)-$th level. Further, we need the following notation from \cite{AnS}.
\begin{definition}\cite{AnS}
Let $F$ be a forest with the root set $V_{root}= \{v_{01},...v_{0r}\}$ and the maximum height of all components is $l-1$.  Then the sequence
$$\pi=(V_0,...,V_{l-1}) $$ is called
the {\it leveled degree sequence} of $F$, if
 $V_i$ is the non-increasing sequence formed by the degrees of the vertices of $F$ at the $i$-th level for  any $i=0, 1, \cdots, l-1$.
\end{definition}

\begin{definition}
Let $F$ be a forest with the following leveled degree sequence
$$\pi=(V_0,...,V_{l-1}).$$
A well-ordering $\prec$ of the vertices in $F$ is called {\it breadth-first search ordering} (BFS-ordering for short) if the following holds for all vertices $u,~ v $ in the same level:\\
$(1)$ $u \prec v$ implies $d(u)\ge d(v)$;\\
$(2)$ If there are two edges $uu_1 \in E(F)$ and $vv_1\in  E(F )$ such that $u \prec v$, $h(u)=h(u_1)+1$ and $h(v)=h(v_1)+1$, then
$u_1 \prec v_1$.
\end{definition}
Moreover, a forest with BFS-ordering is called {\it level greedy forest}. If the forest is a tree, then it is called {\it level greedy tree}. If $|V_0|=2$ and add an edge to the vertices in $V_0$, then it is called {\it edge-rooted level greedy tree}. If $|V_0|=1$ and $\min\limits_{d\in V_i} \{d\}\ge \max\limits_{d'\in V_{i+1}}\{d'\},~0\le i\le l-2$, then it is called {\it greedy tree}. It is easy, but boring, to check the above definitions is equivalent to the level greedy forest (tree, etc) defined in \cite{AnS}.
For a given tree degree sequence $\pi$, there exists exactly one greedy tree with degree sequence $\pi$. Moreover, this greedy tree is denoted by $T_{\pi}^*$ (see \cite{Zh}).

In addition, we also need the notation of majorization. Let  $\alpha=(x_0,x_1,...,x_{n-1})$ and $\beta=(y_0,y_1,\cdots, y_{n-1}$ be two nonnegative sequences. We arrange the entries of $\pi$  and $\tau$ in nonincreasing order
 $\pi_{\downarrow}= (x_{[0]}, \cdots,  x_{[n-1]})$  and
$\tau_{\downarrow} = (y_{[0]},\cdots,  y_{[n-1]}) $ with $x_{[0]}\ge x_{[1]}\ge \cdots\ge x_{[n-1]}$  and  $y_{[0]}\ge y_{[1]}\ge\cdots\ge y_{[n-1]}$. Then we say that $\alpha$ is  {\it weakly majorized} by $\beta$, denoted by $\alpha\lhd_w \beta$, if
\begin{equation}\label{M-Inequality}
\sum_{i=0}^tx_{[i]}\le \sum_{i=0}^ty_{[i]}\ \ \mbox{for} \ t=0,1, \cdots, n-1.
\end{equation}.
Furthermore, if
$$\sum_{i=0}^{n-1}x_{[i]}= \sum_{i=0}^{n-1}y_{[i]}$$
then $\alpha$ is {\it  majorized} by $\beta$, denoted by $\alpha\lhd\beta$. If some inequality in $(\ref{M-Inequality})$ is strict, then the majorization (weak majorization, respectively) is {\it strict}. For more about the majorization, the reader may refer to \cite{MaO}.

Now we are ready to present the main results of this paper.
\begin{theorem}\label{Maintheorem1}
Let $\pi$ be a tree degree sequence.  Then, for any $T\in \mathcal{T}_\pi$, $$C(k;T)\lhd_w C(k;T_{\pi}^*),$$ where $C_{v}(k;T)$ is  the number of the closed walks of length $k$ starting at vertex $v$ in a tree $T$ and  $C(k;T)\equiv(C_{v}(k;T), v\in V(T))$. In other words,
$$\max_{T\in \mathcal{T}_{\pi}}\max_{U\subseteq V(T), |U|=r}\sum_{v\in U}C_k(v, T)
=\sum_{v\in U^*}C_k(v, T_{\pi}^*),\ \mbox{for}\ r=0, \cdots, n-1,$$
where $U^*$ is the first $r$ vertices in the greedy tree $T_{\pi}^*$ with a BFS-ordering.
\end{theorem}
\begin{theorem}\label{Maintheorem2}
Let $\pi,~\pi'$ be two tree degree sequences with $\pi\lhd_w\pi'$. Then $$C(k;T_{\pi}^*)\lhd_w C(k;T_{\pi'}^*).$$
In other words, if $U_{\pi}^*$  and $U_{\pi'}^*$ are the first $r$ vertices in the greedy tree $T_{\pi}^*$  and $T_{\pi'}^*$ with the BFS-ordering respectively, then
$$\sum_{v\in U_{\pi}^*}C_k(v, T_{\pi}^*)\le \sum_{v\in U_{\pi'}^*}C_k(v, T_{\pi'}^*).$$
\end{theorem}

\section{The proof of Theorem \ref{Maintheorem1}}
Let $F$ be a rooted forest.   Denote by $\mathcal{C}_v(k;F)$  the set of closed walks of length $k$ starting at $v$ in $T$. Clearly, $|\mathcal{C}_v(k;F)|=C_v(k;F)$.
  If $W=w_1e_1w_2e_2...e_{k-1}w_k$ be a walk in a rooted forest, $(i_1,i_2,...,i_k)$ is called {\it level sequence} of $W$ if $w_t$ is in the $i_t$-th level for $1\le t\le k$. Denote by $C_v(i_1,...,i_k;F)$ the number of closed walks of length $k$ starting at $v$ in $F$ and the level sequences of the closed walks are $(i_1,i_2,...,i_k)$. Denote by $\mathcal{S}(v_j^i,k;F)$ the set of level sequences of walks of length $k$ in $F$ starting at vertex $v_j^i$  in the $i-$ level and by $\mathcal{S}_i(k;F)=\bigcup\limits_{p=1}^{l_i}\mathcal{S}(v_p^i,k;F)$, where $U_i=\{
  v_1^i, \ddots, v_{l_i}^i\}$ is the set of all vertices in the $i$-th level. Denote by $C_{v,w}(k;F)$ be the number of closed walks of length $k$ starting from the edge $vw$ in $F$.  For $v\in V(F)$,  denote the father of $v$ by $f(v)$ if $v$ has father.  Moreover, denote by $\mathcal{F}_{\pi}$ the set of all rooted forests with leveled degree sequence $\pi$. Before presenting the proof of Theorem~\ref{Maintheorem1}, we need some Lemmas.
\begin{lemma}\cite{AnS}\label{1.3}
Let $F\in\mathcal{F}_{\pi}$ for some leveled degree sequence $\pi$ of a vertex-rooted forest and $G=F_{\pi}$
 be the associated leveled greedy forest. Let $v_1^i,\cdots,v_{l_i}^i$ and $g_1^i,\cdots,g_{l_i}^i$ be the vertices of $F$ and $G$ at the $i$-th
 level, respectively. Then the following relations hold for all $i$:
 \begin{eqnarray*}
(C_{v_1^i}(i_1,\cdots,i_l;F),\cdots,C_{v_{l_i}^i}(i_1,\cdots,i_l;F))\lhd_w (C_{g_1^i}(i_1,\cdots,i_l;G),\cdots,C_{g_{l_i}^i}(i_1,\cdots,i_l;G))
\end{eqnarray*}
and
\begin{eqnarray*}
C_{g_1^i}(i_1,\cdots,i_l;G)\geq \cdots\geq C_{g_{l_i}^i}(i_1,\cdots,i_l;G)
\end{eqnarray*}
\end{lemma}

\begin{lemma}
 \label{lemma1.4}
Let $T\in\mathcal{F}_{\pi}$ for some leveled degree sequence $\pi$ of a vertex-rooted forest and $G=F_{\pi}$
 be the associated leveled greedy forest. Let $v_1^i,\cdots,v_{l_i}^i$ and $g_1^i,\cdots,g_{l_i}^i$ be the vertices of $F$ and $G$ at the $i$-th
 level, respectively. Then the following relations hold for all $i$:
 \begin{eqnarray}
(C_{v_1^i}(k;F),\cdots,C_{v_{l_i}^i}(k;F))\lhd_w (C_{g_1^i}(k;G),\cdots,C_{g_{l_i}^i}(k;G))
\end{eqnarray}
and
\begin{eqnarray}
C_{g_1^i}(k;G)\geq C_{g_2^i}(k;G)\ge\cdots\ge C_{g_{l_i}^i}(k;G).
\end{eqnarray}
\end{lemma}
\begin{proof}
Since
$$C_{v_j^i}(k;F)=\sum\limits_{i_1i_2\cdots i_{k+1}\in \mathcal{S}(v_j^i,k;F)}C_{v_{j}^i}(i_1,\cdots,i_{k+1};F)=\sum\limits_{i_1i_2\cdots i_{k+1}\in \mathcal{S}_i(k;F)}C_{v_{j}^i}(i_1,\cdots,i_{k+1};F),$$
we have
\begin{eqnarray*}
\sum\limits_{j=1}^tC_{v_j^i}(k,F)
&=&\sum\limits_{i_1i_2\cdots i_{k+1}\in \mathcal{S}_i(k;F)}\sum\limits_{j=1}^tC_{v_{j}^i}(i_1,\cdots,i_{k+1};F)\\
&\leq&\sum\limits_{i_1i_2\cdots i_{k+1}\in \mathcal{S}_i(k;F)}\sum\limits_{j=1}^tC_{g_{j}^i}(i_1,\cdots,i_{k+1};G)~~(\mbox{by~ Lemma~\ref{1.3}})\\
&\leq&\sum\limits_{i_1i_2\cdots i_{k+1}\in \mathcal{S}_i(k;G)}\sum\limits_{j=1}^tC_{g_{j}^i}(i_1,\cdots,i_{k+1};G)~~( \mbox{by}~ \mathcal{S}_i(k;F)\subset\mathcal{S}_i(k;G))\\
&=&\sum\limits_{j=1}^t\sum\limits_{i_1i_2\cdots i_{k+1}\in \mathcal{S}_i(k;G)}C_{g_{j}^i}(i_1,\cdots,i_{k+1};G)\\
&=&\sum\limits_{j=1}^tC_{g_{j}^i}(k;G)
\end{eqnarray*}
for $1\leq t\leq l_i$.
And
\begin{eqnarray*}
C_{g_j^i}(k;G)
&=&\sum\limits_{i_1i_2\cdots i_{k+1}\in \mathcal{S}(g_j^i,k;G)}C_{g_{j}^i}(i_1,\cdots,i_{k+1};G)\\
&\geq&\sum\limits_{i_1i_2\cdots i_{k+1}\in \mathcal{S}(g_{j+1}^i,k;G)}C_{g_{j+1}^i}(i_1,\cdots,i_{k+1};G)~~(\mbox{by Lemma~\ref{1.3}})\\
&=&C_{g_{j+1}^i}(k;G)
\end{eqnarray*}
for $1\leq j\leq l_i-1$, since $\mathcal{S}(g_{j}^i,k;G)\supseteq \mathcal{S}(g_{j+1}^i,k;G)$.
\end{proof}
Denote by $\widehat{C}_{v^i_j,v^{i+1}_{j'}}(2k;F),~\widehat{C}_{v^i_j}(2k;F)$ the number of closed walks of length $2k$ in $F$ starting from $v^i_jv^{i+1}_{j'}$ and $v_j^i$, respectively, and the level sequences of the closed walks do not contain pairs $(0,0)$ and $(i,i-1)$ if $i>0$.
\begin{lemma}\label{1.5}
Let $F\in\mathcal{F}_{\pi}$ for some leveled degree sequence $\pi$ of a vertex-rooted forest, and let $G=F_{\pi}$
 be the associated leveled greedy forest. Let $v_1^i,\cdots,v_{l_i}^i$ and $g_1^i,\cdots,g_{l_i}^i$ be the vertices of $F$ and $G$ at the $i$-th
 level, respectively. Then the following relations hold for all $i$:
$$(\widehat{C}_{f(v_1^i),v_1^i}(2k;F),\cdots,\widehat{C}_{f(v_{l_i}^i),v_{l_i}^i}(2k;F))
\lhd_w(\widehat{C}_{f(g_1^i),g_1^i}(2k;G),\cdots,\widehat{C}_{f(g_{l_i}^i),g_{l_i}^i}(2k;G))$$
and
\begin{eqnarray*}
\widehat{C}_{f(g_1^i),g_1^i}(2k;G)&\ge&\widehat{C}_{f(g_2^i),g_2^i}(2k;G)\ge \cdots \ge \widehat{C}_{f(g_{l_i}^i),g_{l_i}^i}(2k;G).
\end{eqnarray*}
\end{lemma}
\begin{proof}
we use induction on $k$. If $k=1$, then it is easy to find the assertion holds. Suppose that the assertion holds for the number not more than $k~(k\ge 1)$. Since
$$\widehat{C}_{f(v_j^i),v_j^i}(2k+2;F)=\sum_{t=0}^k\widehat{C}_{v_j^i}(2t;F)\cdot\widehat{C}_{f(v_j^i)}(2k-2t;T),$$
by Lemma $8$ in \cite{AnS} and Lemma~\ref{1.3}, we have
\begin{eqnarray*}
\sum_{j=1}^m\widehat{C}_{f(v_j^i),v_j^i}(2k+2;F)
\leq \sum_{t=0}^k\sum_{j=1}^m\widehat{C}_{g_j^i}(2t;G)\widehat{C}_{f(g_j^i)}(2k-2t;G)
=\sum_{j=1}^m\widehat{C}_{f(g_j^i),g_j^i}(2k+2;G)
\end{eqnarray*}
and
\begin{eqnarray*}
\widehat{C}_{f(g_j^i),g_j^i}(2k+2;G)
\geq\sum_{t=0}^k\widehat{C}_{g_{j+1}^i}(2t;G)\widehat{C}_{g_{j+1}^i}(2k-2t;G)
=\widehat{C}_{f(g_{j+1}^i),g_{j+1}^i}(2k+2;G).
\end{eqnarray*}
This completes the proof.
\end{proof}
\begin{lemma}\cite{AnS}\label{Root1}
Let $\pi$ be a leveled degree sequence of an edge-rooted tree and $G=T_{\pi}$ be the associated edge-rooted greedy tree. For any element $T\in \mathcal{T}_{\pi}$, we have
$$C_{v_1^0,v_2^0}(k;T)=C_{v_2^0,v_1^0}(k;T)\leq C_{g_1^0,g_2^0}(k;G)=C_{g_2^0,g_1^0}(k;G)$$
for any nonnegative integer $k$, where $v_1^0,~v_2^0$ and $g_1^0,~g_2^0$ are the roots of $T$ and $G$, respectively.
\end{lemma}

\begin{lemma}\label{Root2}
Let $F\in\mathcal{F}_{\pi}$ for some leveled degree sequence $\pi$ of an edge-rooted forest and $G=F_{\pi}$
 be the associated leveled greedy forest. Let $v_1^i,\cdots,v_{l_i}^i$ and $g_1^i,\cdots,g_{l_i}^i$ be the vertices of $T,~G$ at the $i$-th
 level, respectively. Then the following relations hold for all $i$:
$$(C_{f(v_1^i),v_1^i}(2k;F),\cdots,C_{f(v_{l_i}^i),v_{l_i}^i}(2k;F))
\lhd_w(C_{f(g_1^i),g_1^i}(2k;G),\cdots,C_{f(g_{l_i}^i),g_{l_i}^i}(2k;G))$$
and
$$C_{f(g_1^i),g_1^i}(2k;G)\ge C_{f(g_2^i),g_2^i}(2k;G)\geq \cdots\ge C_{f(g_{l_i}^i),g_{l_i}^i}(2k;G)).$$
\end{lemma}
\begin{proof}
Induction on $k$, if $k=1$, then it is easy to find the assertion holds. Suppose that the assertion holds for the number not more than $k~(k\ge 1)$. Without loss of generality, we can suppose that $C_{f(v_1^i),v_1^i}(2k+2; F)\ge \cdots\ge C_{f(v_{l_i}^i),v_{l_i}^i}(2k+2;F)$, otherwise, we can change the label of the vertex in $T$. We divide the following two cases:\\
{\bf Case 1}: $i=1$, we need to prove
$$(C_{f(v_1^1),v_1^1}(2k+2;F), \cdots,C_{f(v_{l_1}^1),v_{l_1}^1}(2k+2;F))
\lhd_w(C_{f(g_1^1),g_1^1}(2k+2;G),\cdots,C_{f(g_{l_1}^1),g_{l_1}^1}(2k+2;G))$$
and
$$C_{f(g_1^1),g_1^1}(2k+2;G)\ge C_{f(g_2^1),g_2^1}(2k+2;G)\geq \cdots\ge C_{f(g_{l_i}^1),g_{l_1}^1}(2k+2;G)).$$
If $f(v_j^1)=v_1^0$, then
$$C_{v_1^0,v_j^1}(2k+2;T)=\sum_{t=1}^k\widehat{C}_{v_1^0,v_j^1}(2t;T)\cdot C_{v_1^0,v_2^0}(2k+2-2t;T)
+\widehat{C}_{v_1^0,v_j^1}(2k+2;T).$$
If $f(v_j^1)=v_2^0$, then
$$C_{v_2^0,v_j^1}(2k+2;T)=\sum_{t=1}^k\widehat{C}_{v_2^0,v_j^1}(2t;T)\cdot C_{v_2^0,v_1^0}(2k+2-2t;T)
+\widehat{C}_{v_2^0,v_j^1}(2k+2;T).$$
By the Lemmas~\ref{1.5}, ~\ref{Root1} and the induction hypothesis, we have
\begin{eqnarray*}
&&\sum_{j=1}^mC_{f(v_j^1),v_j^1}(2k+2;T)\\
&&=\sum_{j=1}^m[\sum_{t=1}^k\widehat{C}_{f(v_j^1),v_j^1}(2t;T)\cdot C_{v_1^0,v_2^0}(2k+2-2t;T)+\widehat{C}_{f(v_j^1),v_j^1}(2k+2;T)]\\
&&\le\sum_{j=1}^m[\sum_{t=1}^k\widehat{C}_{f(g_j^1),g_j^1}(2t;G)\cdot C_{g_1^0,g_2^0}(2k+2-2t;G)+\widehat{C}_{f(g_j^1),g_j^1}(2k+2;G)]\\
&&=\sum_{j=1}^mC_{f(g_j^1),g_j^1}(2k+2;G)
\end{eqnarray*}
and
\begin{eqnarray*}
&&{C}_{f(g_j^1),g_j^1}(2k+2;G)\\
&&=\sum_{t=1}^k\widehat{C}_{f(g_j^1),g_j^1}(2t;G)\cdot C_{g_1^0,g_2^0}(2k+2-2t;G)+\widehat{C}_{f(g_j^1),g_j^1}(2k+2;G)\\
&&\ge\sum_{t=1}^k\widehat{C}_{f(g_{j+1}^1),g_{j+1}^1}(2t;G)\cdot C_{g_1^0,g_2^0}(2k+2-2t;G)+\widehat{C}_{f(g_{j+1}^1),g_{j+1}^1}(2k+2;G)\\
&&={C}_{f(g_{j+1}^1),g_{j+1}^1}(2k+2;G)
\end{eqnarray*}
{\bf Case 2}: $i\ge 2$.
Since
$$C_{f(v_j^i),v_j^i}(2k+2;F)=\sum_{t=1}^k\widehat{C}_{f(v_j^i),v_j^i}(2t;F)\cdot C_{f(v_i^j),f^2(v_j^i)}(2k+2-2t;F)+\widehat{C}_{f(v_j^i),v_j^i}(2k+2;F),$$
by  Lemmas~\ref{1.5}, \ref{Root1} and the induction hypothesis, we have
\begin{eqnarray*}
&&\sum_{j=1}^mC_{f(v_j^i),v_j^i}(2k+2;T)\\
&&=\sum_{t=1}^k\sum_{j=1}^m\widehat{C}_{f(v_j^i),v_j^i}(2t;T)
C_{f(v_j^i),f^2(v_j^i)}(2k+2-2t;T)+\sum_{j=1}^m\widehat{C}_{f(v_j^i),v_j^i}(2k+2;T)\\
&&\leq\sum_{t=1}^k\sum_{j=1}^m\widehat{C}_{f(g_j^i),g_j^i}(2t;G)
C_{f(g_j^i),f^2(g_j^i)}(2k+2-2t;G)+\sum_{j=1}^m\widehat{C}_{f(g_j^i),g_j^i}(2k+2;G)\\
&&=\sum_{j=1}^m[\sum_{t=1}^k\widehat{C}_{f(g_j^i),g_j^i}(2t;G)
C_{f(g_j^i),f^2(g_j^i)}(2k+2-2t;G)+\widehat{C}_{f(g_j^i),g_j^i}(2k+2;G)]\\
&&=\sum_{j=1}^mC_{f(g_j^i),g_j^i}(2k+2;G)
\end{eqnarray*}
and
\begin{eqnarray*}
&&\widehat{C}_{f(g_j^i),g_j^i}(2k+2;G)\\
&&=\sum_{t=1}^k\widehat{C}_{f(g_j^i),g_j^i}(2t;G)\widehat{C}_{f(g_j^i),f^2(g_j^i)}(2k+2-2t;G)+\widehat{C}_{f(g_j^i),g_j^i}(2k+2;G)\\
&&\geq\sum_{t=1}^k\widehat{C}_{f(g_{j+1}^i),g_{j+1}^i}(2t;G)\widehat{C}_{f(g_{j+1}^i),f^2(g_{j+1}^i)}(2k+2-2t;G)+\widehat{C}_{f(g_{j+1}^i),g_{j+1}^i}(2k+2;G)\\
&&=\widehat{C}_{f(g_{j+1}^i),g_{j+1}^i}(2k+2;G).
\end{eqnarray*}
This completes the proof.
\end{proof}
\begin{lemma}\label{Edgeroot}
Let $T\in\mathcal{F}_{\pi}$ for some leveled degree sequence $\pi$ of an edge-rooted forest, and let $G=F_{\pi}$
 be the associated leveled greedy forest. Let $v_1^i,\cdots,v_{l_i}^i$ and $g_1^i,\cdots,g_{l_i}^i$ be the vertices of $F$ and $G$ at the $i$-th
 level, respectively. Then
\begin{eqnarray*}
(C_{v_1^i}(2k;F),\cdots,C_{v_{l_i}^i}(2k;F))\lhd_w (C_{g_1^i}(2k;G),\cdots,C_{g_{l_i}^i}(2k;G))
\end{eqnarray*}
and
$$C_{g_1^i}(2k;G)\geq C_{g_2^i}(2k;G)\ge \cdots\geq  C_{g_{l_i}^i}(2k;G).$$
\end{lemma}
\begin{proof}
Induction on $k$. If $k=1$, then it is easy to find the assertion holds. Suppose that the assertion holds for the number not more than $k~(k\ge 1)$. Without loss of generality, we can suppose that $C_{v_1^i}(2k+2;T)\ge \cdots\ge C_{v_{l_i}^i}(2k+2;T)$, otherwise, we can change the label of the vertex in $T$. We divide the following two cases:\\
{\bf Case 1}: $i=0.$
Since
$$C_{v_1^0}(2k+2;T)=\sum_{t=1}^k\widehat{C}_{v_1^0}(2t;T)C_{v_1^0,v_2^0}(2k+2-2t;T)
+\widehat{C}_{v_1^0}(2k+2;T)+C_{v_1^0,v_2^0}(2k+2;T)$$
$$C_{v_2^0}(2k+2;T)=\sum_{t=1}^k\widehat{C}_{v_2^0}(2t;T)C_{v_2^0,v_1^0}(2k+2-2t;T)
+\widehat{C}_{v_2^0}(2k+2;T)+C_{v_2^0,v_1^0}(2k+2;T),$$
 by Lemmas~\ref{lemma1.4}, ~\ref{Root2} and the induction hypothesis, we have
\begin{eqnarray*}
&&C_{v_1^0}(2k+2;T)\\
&&=\sum_{t=1}^k\widehat{C}_{v_1^0}(2t;T)C_{v_1^0,v_2^0}(2k+2-2t;T)
+\widehat{C}_{v_1^0}(2k+2;T)+C_{v_1^0,v_2^0}(2k+2;T)\\
&&\leq\sum_{t=1}^k\widehat{C}_{g_1^0}(2t;G)C_{g_1^0,g_2^0}(2k+2-2t;G)
+\widehat{C}_{g_1^0}(2k+2;G)+C_{g_1^0,g_2^0}(2k+2;G)\\
&&=C_{g_1^0}(2k+2;G)
\end{eqnarray*}
and
\begin{eqnarray*}
&&C_{v_1^0}(2k+2;T)+C_{v_2^0}(2k+2;T)\\
&&=\sum_{j=1}^2\sum_{t=1}^k\widehat{C}_{v_j^0}(2t;T)C_{v_{j}^0,v_{j'}^0}(2k+2-2t;T)+\sum_{j=1}^2\widehat{C}_{v_j^0}(2k+2;T)+\sum_{j=1}^2C_{v_j^0,v_{j'}^0}(2k+2;T)\\
&&\leq\sum_{t=1}^k\sum_{j=1}^2\widehat{C}_{g_j^0}(2t;G)C_{g_j^0,g_{j'}^0}(2k+2-2t;G)
+\sum_{j=1}^2\widehat{C}_{g_j^0}(2k+2;G)+\sum_{j=1}^2C_{g_j^0,g_{j'}^0}(2k+2;G)\\
&&=C_{g_1^0}(2k+2;G)+C_{g_2^0}(2k+2;G).
\end{eqnarray*}

{\bf Case 2}: $i\geq2$.
Since
$$C_{v_j^i}(2k+2;F)=C_{f(v_j^i),v_j^i}(2k+2;F)+\sum_{t=1}^k\widehat{C}_{v_j^i}(2t;F)
C_{f(v_j^i),v_j^i}(2k+2-2t;F)
+\widehat{C}_{v_j^i}(2k+2;F),$$
 by Lemmas~\ref{lemma1.4}, ~\ref{Root2} and the induction hypothesis, we have
\begin{eqnarray*}
&&\sum_{j=1}^mC_{v_j^i}(2k+2;F)\\
&&=\sum_{j=1}^mC_{f(v_j^i),v_j^i}(2k+2;F)+\sum_{t=1}^k\sum_{j=1}^m\widehat{C}_{v_j^i}(2t;F)
C_{f(v_j^i),v_j^i}(2k+2-2t;F)
+\sum_{j=1}^m\widehat{C}_{v_j^i}(2k+2;F)\\
&&\leq\sum_{j=1}^mC_{f(g_j^i),g_j^i}(2k+2;G)+\sum_{t=1}^k\sum_{j=1}^m\widehat{C}_{g_j^i}(2t;G)C_{f(g_j^i),g_j^i}(2k+2-2t;G)
+\sum_{j=1}^m\widehat{C}_{g_j^i}(2k+2;G)\\
&&=\sum_{j=1}^mC_{g_j^i}(2k+2;G)
\end{eqnarray*}
and
\begin{eqnarray*}
&&C_{g_j^i}(2k+2;G)\\
&&=C_{f(g_j^i),g_j^i}(2k+2;G)+\sum_{t=1}^k\widehat{C}_{g_j^i}(2t;G)C_{f(g_j^i),g_j^i}(2k+2-2t;G)
+\widehat{C}_{g_j^i}(2k+2;G)\\
&&\geq C_{f(g_{j+1}^i),g_{j+1}^i}(2k+2;G)+\sum_{t=1}^k\widehat{C}_{g_{j+1}^i}(2t;G)C_{f(g_{j+1}^i),g_{j+1}^i}(2k+2-2t;G)
+\widehat{C}_{g_{j+1}^i}(2k+2;G)\\
&&=C_{g_{j+1}^i}(2k+2;G).
\end{eqnarray*}
This completes the proof.
\end{proof}

\begin{theorem}\label{2keven}
Let $T\in\mathcal{T}_{\pi}$ for the leveled degree sequence $\pi$, and let $G=T_{\pi}^*$
 be the associated greedy tree. Then for any positive integer $k$,
\begin{eqnarray}
C(2k;T)\lhd_w C(2k;G)
\end{eqnarray}
Moreover, the majorization is strict for sufficiently large even $k$ if $T$ and $T_{\pi}^*$ are not isomorphic
\end{theorem}
\begin{proof}
If possible to choose an edge or a vertex as root such that $T$ is not level greedy, then choose the edge or vertex as root to get
$T_1$ being the level greedy tree with the same leveled degree sequence $\pi$. We iterate this process: if
an edge or a vertex root can be chosen such that $T_l$ is not level greedy, choose the edge or vertex as root to get a
level greedy tree, which we denote by $T_{l+1}$. Then, by Theorem~$20$ in \cite{AnS}, no infinite loops are possible in this process.
By Lemma~\ref{lemma1.4}, Lemma~\ref{Edgeroot} and Theorems~$15,~19$ in \cite{AnS}, we have
\begin{eqnarray*}
C(2k;T_l)\lhd_w C(2k;T_{l+1})
\end{eqnarray*}
Moreover, the Majorization is strict for sufficiently large $k$.
Hence there exists an integer $m$ such  that $T_m$ is level greedy with respect to any choice of vertex
or edge root. This tree $T_m$  satisfies the ※semi-regular§ property defined in \cite{SzW}, and hence it is a
greedy tree. This completes the proof.
\end{proof}
Since the number of closed walks in a tree with length odd is zero, by Theorem~\ref{2keven}, Theorem~\ref{Maintheorem1} holds. Therefore, we finish the proof of Theorem~\ref{Maintheorem1}.
\section{The Proof of Theorem~\ref{Maintheorem2}}
In order to prove Theorem~\ref{Maintheorem2}, we need the following Lemma.
\begin{lemma}\label{Majorizationlemma}
Let $\alpha=(a_1,...,a_n)$, $\beta=(b_1,...,b_n)$, $V_1\subset \{1, \cdots, n\}$ and $\varphi$ be a bijective map from $\{1, \cdots, n\}$ to $\{1,\cdots, n\}$ such that $(1)$ $\varphi(V_1)\cap V_1=\phi$.
$(2)$ $a_i\le b_i$ for $i\notin V_1$; $a_i+a_{\varphi(i)}\le b_i+b_{\varphi(i)}$ and $a_i\le a_{\varphi(i)}$  for $i\in V_1$. Then $\alpha\lhd_w \beta$.
\end{lemma}
\begin{proof}
Induction on $|V_1|$ which is the size of $V_1$, let $k=|V_1|$. If $k=1$, then the assertion holds by considering that the sum of the first $l$ largest elements of $\alpha$ contains $a_{\varphi(i)},~i\in V_1$ or does not contain. Next suppose the assertion holds for $k>2$, we will prove that the assertion holds for $k+1$. Let $i_0\in V_1$, and $\alpha'=(a'_1,...,a'_n)^T$ where $a'_i=a_i$ for $i\neq i_0,\varphi(i_0)$ and $a'_i=b_i$ for $i=i_0,\varphi(i_0)$. By induction,
$$\alpha\lhd_wa'\lhd_w\beta.$$
This completes the proof.
\end{proof}
Denoted by $\mathcal{C}^e_u(k;G)$ be the set of the closed walks of length $k$ in $G$ starting at $u$ and going through $e$. The cardinality of $\mathcal{C}^e_u(k;G)$ is denoted by $C^e_u(k;G)$.
\begin{theorem}\label{VertexM}
Let $D$ be a leveled degree sequence of rooted tree and $e=xx_1\in E(T_{\pi}^*)$, $B$ is a branch of the level greedy tree $G=T_{\pi}^*$ by deleting the edge $e$, which does not contain the root. Let $T=G-xx_1+x'x_1$ where $x,~x'$ are in the same level(see $Fig.1$), then $C(2k;G)\lhd_w C(2k;T)$.
\end{theorem}
$$
\begin{tikzpicture}[scale=0.75]
\coordinate[label=$w$] (I) at (5cm,6cm);
\coordinate[label=$y'$] (I) at (4cm,5cm);
\coordinate[label=$y$] (I) at (6cm,5cm);
\coordinate[label=$x$] (I) at (8cm,4cm);
\coordinate[label=$x_1$] (I) at (9.4cm,3.2cm);
\coordinate[label=$x'$] (I) at (2cm,4cm);
\coordinate[label=$e$] (I) at (8.5cm,3.6cm);
\coordinate[label=$C_0$] (I) at (5cm,2.5cm);
\coordinate[label=$C_1$] (I) at (6cm,2.5cm);
\coordinate[label=$C_i$] (I) at (8cm,2.5cm);
\coordinate[label=$B$] (I) at (9cm,2.45cm);
\coordinate[label=$C_1'$] (I) at (4cm,2.45cm);
\coordinate[label=$C_i'$] (I) at (2cm,2.45cm);
\coordinate[label=$\cdots$] (I) at (3cm,2.4cm);
\coordinate[label=$\cdots$] (I) at (7cm,2.4cm);
\coordinate[label=$G$] (I) at (5cm,1.25cm);
\fill[black]
(5cm,6cm) circle(0.6mm)
(4cm,5cm) circle(0.6mm)
(6cm,5cm) circle(0.6mm)
(2cm,4cm) circle(0.6mm)
(8cm,4cm) circle(0.6mm)
(9cm,3.5cm) circle(0.6mm);
\draw[-]
(5cm,6cm)--(7cm,8cm)
(5cm,6cm)--(4cm,5cm)
(5cm,6cm)--(6cm,5cm)
(8cm,4cm)--(9cm,3.5cm);
\draw[dashed]
(2cm,4cm)--(4cm,5cm)
(8cm,4cm)--(6cm,5cm)
(5cm,6cm)--++(0.45cm,-3.5cm)--++(-0.9cm,0cm)--cycle
(4cm,5cm)--++(0.45cm,-2.5cm)--++(-0.9cm,0cm)--cycle
(6cm,5cm)--++(0.45cm,-2.5cm)--++(-0.9cm,0cm)--cycle
(2cm,4cm)--++(0.45cm,-1.5cm)--++(-0.9cm,0cm)--cycle
(8cm,4cm)--++(0.45cm,-1.5cm)--++(-0.9cm,0cm)--cycle
(9cm,3.5cm)--++(0.45cm,-1cm)--++(-0.9cm,0cm)--cycle
(1.3cm,7cm)--(9.7cm,7cm)--(9.7cm,1cm)--(1.3cm,1cm)--(1.3cm,7cm);
\end{tikzpicture}
\hskip 1cm
\begin{tikzpicture}[scale=0.75]
\coordinate[label=$w$] (I) at (8cm,6cm);
\coordinate[label=$y'$] (I) at (7cm,5cm);
\coordinate[label=$y$] (I) at (9cm,5cm);
\coordinate[label=$x$] (I) at (11cm,4cm);
\coordinate[label=$x_1$] (I) at (3.6cm,3.2cm);
\coordinate[label=$x'$] (I) at (5cm,4cm);
\coordinate[label=$e'$] (I) at (4.3cm,3.6cm);
\coordinate[label=$C_0$] (I) at (8cm,2.5cm);
\coordinate[label=$C_1$] (I) at (9cm,2.5cm);
\coordinate[label=$C_i$] (I) at (11cm,2.5cm);
\coordinate[label=$B$] (I) at (4cm,2.45cm);
\coordinate[label=$C_1'$] (I) at (7cm,2.45cm);
\coordinate[label=$C_i'$] (I) at (5cm,2.45cm);
\coordinate[label=$\cdots$] (I) at (6cm,2.4cm);
\coordinate[label=$\cdots$] (I) at (10cm,2.4cm);
\coordinate[label=$T$] (I) at (8cm,1.25cm);
\fill[black]
(8cm,6cm) circle(0.6mm)
(7cm,5cm) circle(0.6mm)
(9cm,5cm) circle(0.6mm)
(5cm,4cm) circle(0.6mm)
(11cm,4cm) circle(0.6mm)
(4cm,3.5cm) circle(0.6mm);
\draw[-]
(8cm,6cm)--(10cm,8cm)
(8cm,6cm)--(7cm,5cm)
(8cm,6cm)--(9cm,5cm)
(5cm,4cm)--(4cm,3.5cm);
\draw[dashed]
(5cm,4cm)--(7cm,5cm)
(11cm,4cm)--(9cm,5cm)
(8cm,6cm)--++(0.45cm,-3.5cm)--++(-0.9cm,0cm)--cycle
(7cm,5cm)--++(0.45cm,-2.5cm)--++(-0.9cm,0cm)--cycle
(9cm,5cm)--++(0.45cm,-2.5cm)--++(-0.9cm,0cm)--cycle
(5cm,4cm)--++(0.45cm,-1.5cm)--++(-0.9cm,0cm)--cycle
(11cm,4cm)--++(0.45cm,-1.5cm)--++(-0.9cm,0cm)--cycle
(4cm,3.5cm)--++(0.45cm,-1cm)--++(-0.9cm,0cm)--cycle
(3.3cm,7cm)--(11.7cm,7cm)--(11.7cm,1cm)--(3.3cm,1cm)--(3.3cm,7cm);
\end{tikzpicture}$$
$$Fig.1$$

\begin{proof}
Let $w$ be the common ancestor of $x$ and $x'$ in $G$, then we can find two vertices $y$ and $y'$ which are adjacent with $w$. And $G-w$(respectively, $T-w$) has two components $G_1,G_2$(respectively, $T_1,T_2$), which contain $y,y'$, respectively. Since $G$ is level greedy tree, there exists a isomorphism $h$ from $T-G_1$ to $T-T_2$ such that $h(x)=x', h(e)=e',h(w)=w$ and keeps the level.

Define
$$F:\mathcal{C}_w(k,G)\longrightarrow \mathcal{C}_w(k,T).$$
Let $W=w_1w_2\cdots w_{k+1}\in \mathcal{C}_w(k,G)$ and $m,M$ be the minimal and maximal index such that $w_m=w_M=w, 1<m\le M<k+1$, if there exist such integers. Then
\begin{itemize} \setlength{\itemsep}{-\itemsep}
\item If $w\notin \{w_2,\cdots ,w_k\}$ and $w_sw_{s+1}\neq e,s=2,3,\cdots ,k$, then $H(W)=W$.
\item If $w\notin \{w_2,\cdots ,w_k\}$ and $w_sw_{s+1}= e,$ for some $s\in\{2,3,\cdots ,k\}$, then \\
    $H(W)=h(w_1)h(w_2)h(w_3)\cdots h(w_k)h(w_{k+1})$.
\item Otherwise, then $H(W)=\phi(w_1\cdots w_{k-1})H(w_m\cdots w_M)\phi(w_{M+1}\cdots w_{k+1})$, \\where $\phi(w_1\cdots w_{k-1})=h(w_1)h(w_2)\cdots h(w_{k-1})$ and $w_sw_{s+1}= e,$ for some \\$s\in \{1,2,..,s-2\}$ and $\phi(w_1\cdots w_{k-1})=w_1w_2\cdots w_{k-1}$ otherwise.
\end{itemize}
That is, break a walk into pieces divided by visiting the vertex $w$, each pieces is either kept the same or replaced by its image under the injective $h$ depending on whether it contains $e$. By the uniqueness of the decomposition of the walks and $h$ is injective, then $H$ is also a injective map. By Lemma~\ref{Majorizationlemma}, It is sufficient to prove the following two cases:\\
{\bf Case 1}: If $w'\notin V(T)-V(T_2)$, then $C_{w'}(k;T)\ge C_{w'}(k;G)$.

It is sufficient if there exists an injective map form $\mathcal{C}_{w'}(k;G)$ to $\mathcal{C}_{w'}(k;T)$. Suppose $W=w_1w_2\cdots w_{k+1}\in \mathcal{C}_{w'}(k;G)$, and $m,M$ defined as before. Define
$$F_1(W)=\phi(w_1\cdots w_{m-1})H(w_m\cdots w_M)\phi(w_{M+1}\cdots w_{k+1}).$$
Next we will verify $F_1$ is injective.\\
Suppose $F_1(W_1)=F_1(W_2), W_1,W_2\in \mathcal{C}_{w'}(k;G)$, then the positions of $w$ in $W_1$ and $W_2$ are same. Let $W_i=w_1^i\cdots w^i_{k+1},i=1,2,$ then
$\phi(w_1^1\cdots w_{m-1}^1)=\phi(w_1^2\cdots w_{m-1}^2)$, $H(w_m^1\cdots w_M^1) =H(w_m^2\cdots w_M^2),~\phi(w_{M+1}^1\cdots w_{k+1}^1)=\phi(w_{M+1}^2\cdots w_{k+1}^2).$ This implies that $W_1=W_2$.\\
{\bf Case 2}: If $w'\in V(T_1)$, then $C_{w'}(k;T)+C_{h(w')}(k;T)\ge C_{w'}(k;G)+ C_{h(w')}(k;G)$.\\
For simplicity, let $u=h(w'),v=w'$, since $T-B=G-B$ and
\begin{eqnarray*}
&&C_u(k;T)=C_u^{e'}(k;T)+(C_u(k;T)-C_u^{e'}(k;T)),\\
&&C_u(k;G)=C_u^{e}(k;G)+(C_u(k;G)-C_u^{e}(k;G)),\\
&&C_v(k;T)=C_v^{e'}(k;T)+(C_v(k;T)-C_v^{e'}(k;T)),\\
&&C_v(k;G)=C_v^{e}(k;G)+(C_v(k;G)-C_v^{e}(k;G)),
\end{eqnarray*}
then
\begin{eqnarray*}
&&C_u(k;T)-C_u(k;G)=C_u^{e'}(k;T)-C_u^{e}(k;G),\\
&&C_v(k;T)-C_v(k;G)=C_v^{e'}(k;T)-C_v^{e}(k;G).
\end{eqnarray*}
Thus $C_{u}(k;T)+C_{v}(k;T)\ge C_{u}(k;G)+C_{v}(k;G)$ holds if and only if $C_u^{e'}(k;T)+C_v^{e'}(k;T)\ge C_u^{e}(k;G)+C_v^{e}(k;G)$.
So it is sufficient if there exists an injective map from $\mathcal{C}_u^{e}(k;G)\cup\mathcal{C}_v^{e}(k;G)$ to $\mathcal{C}_u^{e'}(k;T)\cup \mathcal{C}_v^{e'}(k;T)$. Define the following injective map:
$$F_2:\mathcal{C}_u^{e}(k;G)\cup\mathcal{C}_v^{e}(k;G)\longrightarrow \mathcal{C}_u^{e'}(k;T)\cup \mathcal{C}_v^{e'}(k;T).$$
Let $W=\widetilde{w}W_1wW_2wW_3\widetilde{w},w\notin V(W_1),w\notin V(W_3)$, suppose that the following closed walk has the same form. \\
{\bf Subcase 2.1}: If $W=uW_1wW_2wW_3u\in \mathcal{C}_u^{e}(k;G)$, define $F_2(W)=uW_1H(wW_2w)W_3u$.\\
{\bf Subcase 2.2}: If $W=vW_1wW_2wW_3v\in \mathcal{C}_v^{e}(k;G)$, divide it into the followings:
\begin{itemize} \setlength{\itemsep}{-\itemsep}
\item If $e\notin E(W_1)\cup E(W_3)$, then $F_2(W)=vW_1H(wW_2w)W_3v$.
\item If $e\in E(W_1),e\notin E(W_3)$, then $F_2(W)=uh(W_1)H(wW_{2}w)h(W_3)u$.
\item If $e\notin E(W_1),e\in E(W_3)$, then $F_2(W)=uh(W_1)H(wW_{2}w)h(W_3)u$.
\item If $e\in E(W_1),e\in E(W_3)$, then $F_2(W)=uh(W_1)H(wW_{2}w)h(W_3)u$.
\end{itemize}
Next we verify that $F_2$ is injective. If $W,\widetilde{W}$ in the same case, then $F_2(W)=F_2(\widetilde{W})$ implies that $W=\widetilde{W}$, since $H$ is injective. If $W\in \mathcal{C}_u^{e}(k;G)$, $\widetilde{W}\in \mathcal{C}_v^{e}(k;G)$. Then $F_2(W)\neq F_2(\widetilde{W})$, since they do not have the same initial vertex or $W_1\neq h(\widetilde{W_1})$ or $W_3\neq h(\widetilde{W_3})$ by $e'\notin E(W_1) \cup E(W_3)$ and $e'\in E(h(\widetilde{W_1}))\cup E(h(\widetilde{W_3}))$.
\end{proof}
\begin{theorem}\label{EdgeM}
Let $D$ be a leveled degree sequence of edge rooted tree and $e=xx_1\in E(T_{\pi}^*)$, $B$ is a branch of the level greedy tree $G=T_{\pi}^*$ by deleting the edge $e$, which does not contain the root. Let $T=G-xx_1+x'x_1$ where $x,~x'$ are in the same level(see $Fig.2$), then $C(k;G)\lhd_w C(k;T)$.
\end{theorem}
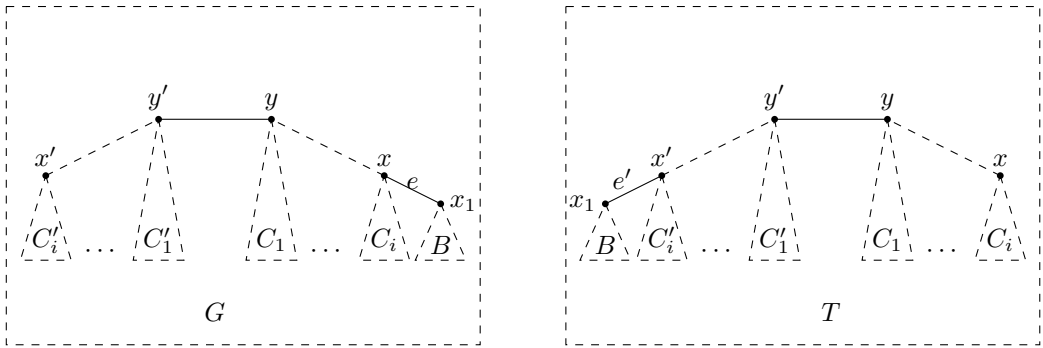
$$
\begin{tikzpicture}[scale=0.75]
\coordinate[label=$y'$] (I) at (4cm,5cm);
\coordinate[label=$y$] (I) at (6cm,5cm);
\coordinate[label=$x$] (I) at (8cm,4cm);
\coordinate[label=$x_1$] (I) at (9.4cm,3.2cm);
\coordinate[label=$x'$] (I) at (2cm,4cm);
\coordinate[label=$e$] (I) at (8.5cm,3.6cm);
\coordinate[label=$C_1$] (I) at (6cm,2.5cm);
\coordinate[label=$C_i$] (I) at (8cm,2.5cm);
\coordinate[label=$B$] (I) at (9cm,2.45cm);
\coordinate[label=$C_1'$] (I) at (4cm,2.45cm);
\coordinate[label=$C_i'$] (I) at (2cm,2.45cm);
\coordinate[label=$\cdots$] (I) at (3cm,2.4cm);
\coordinate[label=$\cdots$] (I) at (7cm,2.4cm);
\coordinate[label=$G$] (I) at (5cm,1.25cm);
\fill[black]
(4cm,5cm) circle(0.6mm)
(6cm,5cm) circle(0.6mm)
(2cm,4cm) circle(0.6mm)
(8cm,4cm) circle(0.6mm)
(9cm,3.5cm) circle(0.6mm);
\draw[-]
(4cm,5cm)--(6cm,5cm)
(8cm,4cm)--(9cm,3.5cm);
\draw[dashed]
(2cm,4cm)--(4cm,5cm)
(8cm,4cm)--(6cm,5cm)
(4cm,5cm)--++(0.45cm,-2.5cm)--++(-0.9cm,0cm)--cycle
(6cm,5cm)--++(0.45cm,-2.5cm)--++(-0.9cm,0cm)--cycle
(2cm,4cm)--++(0.45cm,-1.5cm)--++(-0.9cm,0cm)--cycle
(8cm,4cm)--++(0.45cm,-1.5cm)--++(-0.9cm,0cm)--cycle
(9cm,3.5cm)--++(0.45cm,-1cm)--++(-0.9cm,0cm)--cycle
(1.3cm,7cm)--(9.7cm,7cm)--(9.7cm,1cm)--(1.3cm,1cm)--(1.3cm,7cm);
\end{tikzpicture}
\hskip 1cm
\begin{tikzpicture}[scale=0.75]
\coordinate[label=$y'$] (I) at (7cm,5cm);
\coordinate[label=$y$] (I) at (9cm,5cm);
\coordinate[label=$x$] (I) at (11cm,4cm);
\coordinate[label=$x_1$] (I) at (3.6cm,3.2cm);
\coordinate[label=$x'$] (I) at (5cm,4cm);
\coordinate[label=$e'$] (I) at (4.3cm,3.6cm);
\coordinate[label=$C_1$] (I) at (9cm,2.5cm);
\coordinate[label=$C_i$] (I) at (11cm,2.5cm);
\coordinate[label=$B$] (I) at (4cm,2.45cm);
\coordinate[label=$C_1'$] (I) at (7cm,2.45cm);
\coordinate[label=$C_i'$] (I) at (5cm,2.45cm);
\coordinate[label=$\cdots$] (I) at (6cm,2.4cm);
\coordinate[label=$\cdots$] (I) at (10cm,2.4cm);
\coordinate[label=$T$] (I) at (8cm,1.25cm);
\fill[black]
(7cm,5cm) circle(0.6mm)
(9cm,5cm) circle(0.6mm)
(5cm,4cm) circle(0.6mm)
(11cm,4cm) circle(0.6mm)
(4cm,3.5cm) circle(0.6mm);
\draw[-]
(7cm,5cm)--(9cm,5cm)
(5cm,4cm)--(4cm,3.5cm);
\draw[dashed]
(5cm,4cm)--(7cm,5cm)
(11cm,4cm)--(9cm,5cm)
(7cm,5cm)--++(0.45cm,-2.5cm)--++(-0.9cm,0cm)--cycle
(9cm,5cm)--++(0.45cm,-2.5cm)--++(-0.9cm,0cm)--cycle
(5cm,4cm)--++(0.45cm,-1.5cm)--++(-0.9cm,0cm)--cycle
(11cm,4cm)--++(0.45cm,-1.5cm)--++(-0.9cm,0cm)--cycle
(4cm,3.5cm)--++(0.45cm,-1cm)--++(-0.9cm,0cm)--cycle
(3.3cm,7cm)--(11.7cm,7cm)--(11.7cm,1cm)--(3.3cm,1cm)--(3.3cm,7cm);
\end{tikzpicture}$$
$$Fig.2$$
\begin{proof}
Let $G_1,G_2$(respectively, $T_1,T_2$) be the two components of $G-yy'$(respectively, $T-yy'$), which contain $x,x'$, respectively. Then we define a isomorphism $h$ from $G_1$ to $T_2$, such that $y'=h(y),e'=h(e)$ and keeps the level. Then we define the following injective map:
$$F: \mathcal{C}_{r(G)}(k;G)\longrightarrow \mathcal{C}_{r(T)}(k;T)$$
Let $W=w_1w_2W_1w_2w_1W_2\cdots w_1w_2W_{2k+1}w_2w_1W_{2k+2}w_1$, where $w_1w_2=yy'$ or $y'y$, $w_1w_2,w_2w_1\notin \cup_{i=1}^{2k+2}E(W_i)$ and $w_1\notin \cup_{i=0}^kW_{2i+1}$. If $W_i$ is a empty set, then denote $wW_iw=w$. Define
\begin{itemize} \setlength{\itemsep}{\itemsep}
\item If $w_1w_2=y'y$, then $H(W)=\phi(y'yW_1yy'W_2)\cdots\phi(y'yW_{2k+1}yy'W_{2k+2})y'$.
Where \\$\phi(y'yW_1yy'W_2)=y'yW_1yy'W_2$ if $e\notin W_1$, $\phi(y'yW_1yy'W_2)=y'yy'h(W_1)y'W_2$ otherwise.

\item If $w_1w_2=yy'$, then $H(W)=yy'W_1H(y'yW_2yy'\cdots  y'yW_{2k}yy') \phi(W_{2k+1}y'yW_{2k+2})\\y$,
where $\phi(W_{2k+1}y'yW_{2k+2})y=W_{2k+1}y'yW_{2k+2}y$ if $e\notin W_{2k+2}$, \\$\phi(W_{2k+1}y'yW_{2k+2})y=W_{2k+1}y'h(W_{2k+2})y'y$ otherwise.
\end{itemize}
In words, break a walk into pieces divided by edges $yy',y'y$, each piece is kept the same or replaced by its image under the injective map $h$ if it contains $e$. Since the decomposition of the walks is unique and $h$ is injective, so $H$ is also injective. By Lemma~\ref{Majorizationlemma}, it is sufficient to prove the following. \\
{\bf Case 1}: If $w'\notin V(T)-V(T_1)$, then $C_{w'}(k;T)\ge C_{w'}(k;G)$.

If $w'\notin V(T)-V(G_1)$, then it is sufficient if there exists an injective map $F_1$ from $\mathcal{C}_{w'}(k;T)$ to $\mathcal{ C}_{w'}(k;G)$. Let $W=w_1w_2\cdots w_{k+1}$ and $m'$(respectively, $M'$) be the smallest (respectively, largest) integer such that $w_{m'}w_{m'+1}=y'y$(respectively, $w_{M'-1}w_{M'}=yy'$). Define $$F_1(W)=\phi(w_1w_2\cdots w_{m-1})H(w_mw_{m+1}\cdots w_M)\phi(w_{M+1}\cdots w_{k+1}).$$
Since $H$ is injective, then $F_1$ is also injective.

If $w'\notin V(B)$. Let $m'$(respectively, $M'$) be the smallest (respectively, largest) integer such that $w_{m'}w_{m'+1}=yy'$(respectively, $w_{M'-1}w_{M'}=y'y$). Define $$F_1(W)=\phi(w_1w_2\cdots w_{m'-1})H(w_{m'}w_{m+1}\cdots w_{M'})\phi(w_{M'+1}\cdots w_{k+1}).$$
Since $H$ is injective, then $F_1$ is also injective.\\
{\bf Case 2}: If $w'\in V(T_1)$, then $C_{w'}(k;T)+C_{h(w')}(k;T)\ge C_{w'}(k;G)+ C_{h(w')}(k;G)$.\\
For simplicity, let $u=h(w'),v=w'$, since $T-B=G-B$ and
\begin{eqnarray*}
&&C_u(k;T)=C_u^{e'}(k;T)+(C_u(k;T)-C_u^{e'}(k;T)),\\
&&C_u(k;G)=C_u^{e}(k;G)+(C_u(k;G)-C_u^{e}(k;G)),\\
&&C_v(k;T)=C_v^{e'}(k;T)+(C_v(k;T)-C_v^{e'}(k;T)),\\
&&C_v(k;G)=C_v^{e}(k;G)+(C_v(k;G)-C_v^{e}(k;G)),
\end{eqnarray*}
then
\begin{eqnarray*}
&&C_u(k;T)-C_u(k;G)=C_u^{e'}(k;T)-C_u^{e}(k;G),\\
&&C_v(k;T)-C_v(k;G)=C_v^{e'}(k;T)-C_v^{e}(k;G).
\end{eqnarray*}
Thus $C_{u}(k;T)+C_{v}(k;T)\ge C_{u}(k;G)+C_{v}(k;G)$ holds if and only if $C_u^{e'}(k;T)+C_v^{e'}(k;T)\ge C_u^{e}(k;G)+C_v^{e}(k;G)$.
So it is sufficient if there exists an injective map from $\mathcal{C}_u^{e}(k;G)\cup\mathcal{C}_v^{e}(k;G)$ to $\mathcal{C}_u^{e'}(k;T)\cup \mathcal{C}_v^{e'}(k;T)$. Define the following injective map:
$$F_2:\mathcal{C}_u^{e}(k;G)\cup\mathcal{C}_v^{e}(k;G)\longrightarrow \mathcal{C}_u^{e'}(k;T)\cup \mathcal{C}_v^{e'}(k;T).$$
Let $W=wW_1y'yW_2yy'W_3w,y'y\notin E(W_1),yy'\notin E(W_3)$, suppose that the following closed walk has the same form. \\
{\bf Subcase 2.1}: If $W=uW_1y'yW_2yy'W_3u\in \mathcal{C}_u^{e}(k;G)$, then define $F_2(W)=uW_1H(y'yW_2yy')W_3u$.\\
{\bf Subcase 2.2}: If $W=vW_1y'yW_2yy'W_3v\in \mathcal{C}_v^{e}(k;G)$, then divide it into the followings:
\begin{itemize} \setlength{\itemsep}{-\itemsep}
\item If $e\notin E(W_1)\cup E(W_3)$, then $F_2(W)=vW_1H(y'yW_2yy')W_3v$.
\item If $e\in E(W_1),e\notin E(W_3)$, then $F_2(W)=uh(W_1)y'W_{21}H(y'yW_{22}yy')h(W_3)u$.
\item If $e\notin E(W_1),e\in E(W_3)$, then $F_2(W)=uh(W_1)y'W_{21}H(y'yW_{22}yy')h(W_3)u$.
\item If $e\in E(W_1),e\in E(W_3)$, then $F_2(W)=uh(W_1)y'W_{21}H(y'yW_{22}yy')h(W_3)u$.
\end{itemize}
where $y'W_2=y'W_{21}y'yW_{22},y'y\notin W_{21}$. Next we verify that $F_2$ is injective. If $W,\widetilde{W}$ in the same case, then $F_2(W)=F_2(\widetilde{W})$ implies that $W=\widetilde{W}$, since $H$ is injective. If $W\in \mathcal{C}_u^{e}(k;G)$, $\widetilde{W}\in \mathcal{C}_v^{e}(k;G)$. Then $F_2(W)\neq F_2(\widetilde{W})$,  since they do not have the same initial vertex or $W_1\neq h(\widetilde{W_1})$ or $W_3\neq h(\widetilde{W_3})$ by $e'\notin E(W_1) \cup E(W_3)$ and $e'\in E(h(\widetilde{W_1}))\cup E(h(\widetilde{W_3}))$.
\end{proof}
Now we are ready to prove Theorem~\ref{Maintheorem2}.
\begin{theorem}
Let $\pi=(d_0,d_1,...,d_{n-1})$ and $\pi'=(d_0',d_1'...,d_{n-1})$ be decreasing degree sequences of trees of the same order such
that $\pi\lhd \pi'$. Then for any integer $k>0$ we have
$$C(2k;T_{\pi}^*)\lhd_w C(2k;T_{\pi'}^*).$$
If $\pi\neq \pi'$ and $k>1$, then the majorization is strict.
\end{theorem}
\begin{proof}
If $\pi= \pi'$, the assertion holds. Next suppose $\pi\neq \pi'$, then there exists an integer $i$ such that $l_i\neq l_i'$. Set $\{i:l_i\neq l_i'\}$, by
$\sum\limits_{i=0}^{n-1}l_i=\sum\limits_{i=0}^{n-1}l_i'$, we will find that $\{i:l_i\neq l_i'\}$ has at least two elements. Let $l=\min\{i:l_i\neq l_i'\},~L=\max\{i:l_i\neq l_i'\}$. Then $d_l<d_l',~d_L>d_L'$. Let
$$\pi_1=(d_0,...,d_{l-1},d_{l}+1,...,d_{L}-1,...,d_{n-1}).$$
We will find that $\pi_1$ is decreasing and $\pi\lhd\pi_1\lhd\pi'$. Next consider the two vertices $u$ and $v$ in $T_{\pi}$ such that $d(u)=d_l,~d(v)=d_L$, then divide the following two cases:\\
{\bf Case 1}: If the distance between $u$ and $v$ is even. Let $w$ be the middle vertex in the path from $u$ to $v$, consider $T_{\pi}$ as a rooted tree with root $w$. Then $u,~v$ are in the same level, let $v'$ be a children of $v$, Consider the tree $T=T_{\pi}^*-vv'+uv'$. By Lemma~\ref{lemma1.4} and Lemma~\ref{VertexM}, we have
$$C(2k;T_{\pi}^*)\lhd_w C(2k;T)\lhd_w C(2k;T_{\pi_1}^*).$$
{\bf Case 2}: If the distance between $u$ and $v$ is odd. Let $yy'$ be the middle in the path from $u$ to $v$, consider $T_{\pi}$ as a edge-rooted tree with edge root $yy'$. Then $u,~v$ are in the same level, let $v'$ be a children of $v$, Consider the tree $T=T_{\pi}^*-vv'+uv'$. By Lemma~\ref{Edgeroot} and Lemma~\ref{EdgeM}, we have
$$C(2k;T_{\pi}^*)\lhd_w C(2k;T)\lhd_w C(2k;T_{\pi_1}^*).$$
By the two cases, we find that $C(2k;T_{\pi}^*)\lhd_w C(2k;T_{\pi_1}^*)$. By repeating the above process we can get $\pi=\pi_0,~\pi_1,~\pi_2,~...,~\pi_m=\pi'$ such that $\pi=\pi_0\lhd\pi_1\lhd\pi_2\lhd...\lhd\pi_m=\pi'$ and
$$C(2k;T_{\pi}^*)=C(2k;T_{\pi_0}^*)\lhd_w C(2k;T_{\pi_1}^*)\lhd_w\cdots\lhd_w C(2k;T_{\pi_m}^*)=C(2k;T_{\pi'}^*).$$
This completes the proof.
\end{proof}
\begin{corollary}
For any $n$ vertex tree. Then
$$C(2k;T)\lhd_w C(2k;S_n),$$
for any positive integer $k$, where $S_n$ is a star of order $n$.
\end{corollary}
\begin{corollary}
For any $n$ vertex tree $T$ with maximal degree is $\Delta$. Then
$$C(2k;T)\lhd_w C(2k;T_{\pi}^*),$$
where $\pi=(\Delta,...,\Delta,r,1,...,1),~1\le r< \Delta$, the sum of the elements of $\pi$ is $2n-2$.
\end{corollary}
\begin{corollary}
For any tree $T$ of order $n$ with $s$ leaves. Then
$$C(2k;T)\lhd_w C(2k;T_{\pi}^*),$$
for any positive integer $k$, where $\pi=(s,2,2,...,2,1,1,...,1)$ (the number of $2$ is $n-s-1$, the number of $1$ is $s$).
\end{corollary}
\begin{corollary}
For any tree T of order $n$ with independence number $\alpha\ge n/2$ or with
matching number $n-\alpha \le n/2$. Then
$$C(2k;T)\lhd_w C(2k;T_{\pi}^*),$$
for any positive integer $k$, where $\pi=(\alpha,2,2,...,2,1,1,...,1)$ (the number of $2$ is $n-\alpha-1$, the number of $1$ is $\alpha$).
\end{corollary}
 For a given tree degree sequence $\pi$, we determined the maximum value of  the number of the closed walks of length $k$ starting at any vertex $v\in U\subseteq V(T)$ in any tree $T=(V(T), E(T))$ with $|U|=r$.  It is interesting to determine the minimum value of them.


\begin{thebibliography}{}
\bibitem{An}E.~O.~D.~Andriantiana, Energy, Hosoya index and Merrifield-Simmons index of trees with prescribed degree sequence, {\it Discrete Appl. Math.,} 161 (2013) 724-741.

\bibitem{AnS}E.~O.~D.~Andriantiana, S.~Wagner, Spectral moments of trees with given degree sequence, {\it Linear Algebra Appl.,} 439 (2013) 3980-4002.

\bibitem{BH}T.~B{\i}y{\i}ko\v{g}lu, M.~Hellmuth, J.~Leydold, Largest eigenvalues of the discrete p-Laplacian of trees with degree sequences, {\it E. J. Linear Algebra,} 18 (2009) 202-210.

\bibitem{BJ}T.~B{\i}y{\i}ko\v{g}lu, J.~Leydold, Graphs with given degree sequence and maximal spectral radius, {\it Electron. J. Combin.,} 15 (2008),
Research Paper 119, 9 pp.

\bibitem{bob}B.~Bollob\'{a}s, M.~Tyomkyn, Walks and paths in trees, {\it J. Graph Theory,} 70 (2012) 54-66.

\bibitem{Csikvari}P.~Csikv\'{a}ri, On a poset of trees, {\it Combinatorica,} 30 (2010) 125-137.

\bibitem{CvD}D.~Cvetkovi\'{c}, M.~Doob, H.~Sachs, Spectra of Graphs - Theory and Application, 3rd ed., Johann Ambrosius Barth Verlag, Heidelberg, Leipzig, 1995.

   \bibitem{dress}A.~Dress, S.~Gr\"{u}newald,  I.~Gutman,  M.~Lepovi\'{c},D.~ Vidovi\'{c},  On the number of walks in trees, {\it  MATCH Commun. Math. Comput. Chem.,} 48 (2003) 63-85.
\bibitem{feige}U.~Feige, G.~Kortsarz,  D.~Peleg, The Dense k-Subgraph Problem, {\it Algorithmica,}  29 (2001)  410-421.

\bibitem{FiH}M.~Fischermann, A.~Hoffmann, D.~Rautenbach, L.~Sz\'{e}kely, L. Volkmann, Wiener index versus maximum degree in trees, {\it
Discrete Appl. Math.,} 122 (2002) 127-137.

\bibitem{GuD}I.~Gutman, H.~Deng, S.~Radenkovi\'{c}, The Estrada index: an updated survey, {\it Zb. Rad. (Beogr.),} 14 (2011) 155-174, Selected
topics on applications of graph spectra.

\bibitem{MaO}A.W.~Marshall, I.~Olkin, Inequalities: Theory of Majorization and Its Applications, Academic, New York (1979).

\bibitem{PeG}J.A.~de la Pe\~{n}a, I.~Gutman, J.~Rada, Estimating the Estrada index, {\it Linear Algebra Appl.,} 427 (2007) 70-76.


\bibitem{SzW}L.A.~Sz\'{e}kely, H.~Wang, T.~Wu, The sum of the distances between the leaves of a tree and the ＆semi-regular＊ property, {\it Discrete Math.,} 311 (2011) 1197-1203.
\bibitem{taubig}H.~T\"{a}ubig, J.Weihmann, S. Kosub, R. Hemmecke and E.W. Mayr,
Inequaltiies fo the number of walks in graphs, {\it Algorithmica} 66(2013) 804-828.


\bibitem{W1}H.~Wang, The extremal values of the Wiener index of a tree with given degree sequence, {\it Discrete Appl. Math.,} 156 (2008) 2647-2654.

\bibitem{W2}H.~Wang, Corrigendum: The extremal values of the Wiener index of a tree with given degree sequence, {\it Discrete Appl.
Math.,} 157 (2009) 3754.


\bibitem{Zh}X.-D.~Zhang, The Laplacian spectral radii of trees with degree sequences, {\it Discrete Math.,} 308 (2008) 3143-3150.

  \bibitem{zhangxiang}  X.-D.~Zhang, Q.-Y.~Xiang, L.-Q.~Xu,  R.-Y.~Pan, The Wiener index of
trees with given degree sequences, {\it MATCH Commun. Math. Comput. Chem.,}
60 (2008) 623-644.

\bibitem{ZhZ1}X.-M.~Zhang, X.-D.~Zhang, Trees with given degree sequences that have minimal subtrees, {\it Graphs and Combinatorics,} 31 (2015) 309-318.

\bibitem{ZhZ2}X.-M.~Zhang, X.-D.~Zhang, D.~Gray, H.~Wang, The number of subtrees of trees with given degree sequence, {\it J. Graph Theory,} 73 (2013) 280-295
\end{thebibliography}
\end{document}